\documentclass[10pt]{amsart}
\usepackage[scale=0.7]{geometry}
\usepackage{graphicx,color,amsthm,hyperref,amsmath,amssymb}
\usepackage[utf8]{inputenc}

\newtheorem{theorem}{Theorem}
\newtheorem{lemma}{Lemma}

\newtheorem{proposition}{Proposition}
\newtheorem{problem}{Problem}

\theoremstyle{definition}
\newtheorem{definition}{Definition}
\theoremstyle{remark}

\newcommand{\cei}[1]{\left\lceil#1\right\rceil}

\newcommand{\dep}{\text{dep}}
\newcommand{\tdep}{\text{tdep}}

\title{Depth with respect to a family of convex sets}

\author[Mart\'inez-Sandoval]{Leonardo Mart\'inez-Sandoval}
\address{Dept. of Computer Science, Faculty of Natural Sciences, Ben-Gurion University of the Negev, Beer Sheva, 84105, Israel}
\email{leomtz@im.unam.mx}
\author[Tamam]{Roee Tamam}
\address{Dept. of Computer Science, Faculty of Natural Sciences, Ben-Gurion University of the Negev, Beer Sheva, 84105, Israel}
\email{roeeta@post.bgu.ac.il}
\begin{document}

\begin{abstract}
   	We propose a notion of depth with respect to a finite family $\mathcal{F}$ of convex sets in $\mathbb{R}^d$ which we call $\dep_\mathcal{F}$. We begin showing that $\dep_\mathcal{F}$ satisfies some expected properties for a measure of depth and that this definition is closely related to the notion of depth proposed by J. Tukey. We show that some properties of Tukey depth extend to $\dep_\mathcal{F}$ and we point out some key differences.
	
	We then focus on the following centerpoint-type question: what is the best depth $\alpha_{d,k}$ that we can guarantee under the hypothesis that the family $\mathcal{F}$ is $k$-intersecting? We show a key connection between this problem and a purely combinatorial problem on hitting sets. The relationship is useful in both directions. On the one hand, for values of $k$ close to $d$ the combinatorial interpretation gives a good bound for $k$. On the other hand, for low values of $k$ we can use the classic Rado's centerpoint theorem to get combinatorial results of independent interest. For intermediate values of $k$ we present a probabilistic framework to improve the bounds and illustrate its use in the case $k\approx d/2$. These results can be though of as an interpolation between Helly's theorem and Rado's centerpoint theorem.
	
	As an application of these results we find a Helly-type theorem for fractional hyperplane transversals. We also give an alternative and simpler proof for a transversal result of A. Holmsen.
	
%\subclass{52A35 \and 52A20}
\end{abstract}

\maketitle

\section{Introduction}

In this article we introduce and study the following notion of depth with respect to a family of convex sets in euclidean space.

\begin{definition}
\label{defDepth}
  Let $d$ be a positive integer. Let $\mathcal{F}$ be a finite family of convex sets and $p$ a point on $\mathbb{R}^d$. We define the depth of $p$ with respect to $\mathcal{F}$ as the smallest number of sets from $\mathcal{F}$ intersected by a closed halfspace that contains $p$. We denote this smallest number as $\dep_\mathcal{F}(p)$.
\end{definition}

In Section \ref{secProp} we study some properties of Definition \ref{defDepth} and we relate it to previous notions of depth in the literature. Specifically, we begin by proving that $\dep_\mathcal{F}$ indeed behaves as a depth function. We refer the reader to that section for the precise definitions.

\begin{theorem}
\label{thmBasic}
\begin{itemize}
	\item For any finite family $\mathcal{F}$ of bounded convex sets in $\mathbb{R}^d$ we have that
	
	\[
		\lim_{||x||\to \infty} \dep_\mathcal{F}(x)=0.
	\]
	\item For any family $\mathcal{F}$ of convex sets that is symmetric about the origin, the function $D:\mathbb{R}_{\geq 0}\to\mathbb{Z}$ given by $D(a)=\dep_{\mathcal{F}}(ax)$ is decreasing.
\end{itemize}
\end{theorem}

We continue Section \ref{secProp} by showing how Definition \ref{defDepth} is a natural generalization of the notion of depth introduced by Tukey \cite{Tukey1974}. We prove that some useful properties of Tukey depth are also valid for Definition \ref{defDepth}, but we point out some distinctions. In particular, we study the set of points $C_r(\mathcal{F})$ of depth at least $r$ with respect to $\dep_{\mathcal{F}}$. In contrast to the analogous definition for Tukey depth, this set is not necessarily a polytope. However, we show that it is always convex.

\begin{theorem}
\label{thmRCen}
  Let $\mathcal{F}$ be a family of convex sets in $\mathbb{R}^d$ and $r$ a non-negative integer. Then the $r$-center $C_r(\mathcal{F})$ is convex.
\end{theorem}

In Section \ref{secCPT} we focus our attention to centerpoint results for Definition \ref{defDepth}. We show that the intersection pattern of the family $\mathcal{F}$ plays a key role. We study families of convex sets in which each family of at most $k$ elements has a non-empty intersection ($k$-intersecting families). In this sense, this section can be though of as an interpolation between the celebrated theorems by Rado and Helly \cite{Helly1923,Matousek2002,Rado1946}. 

More precisely, we define $\alpha_{d,k}$ as the maximum number such that any finite family $\mathcal{F}$ of $k$-intersecting convex sets on $\mathbb{R}^d$ has a point of depth at least $\alpha_{d,k}\cdot |\mathcal{F}|$. As an easy consequence of the results in Section \ref{secProp} we obtain a basic bound on $\alpha_{d,k}$.

\begin{theorem}
\label{thmAsy}
	For a fixed $k$, the value of $\alpha_{d,k}$ is in $\Omega\left(\frac{1}{\sqrt[k]{d+1}}\right)$.
\end{theorem}

In order to improve this bound, we establish a key connection that shows that determining the precise value of $\alpha_{d,k}$ is deeply related to a purely combinatorial problem on hitting sets. As a reminder, for a family of sets $\mathcal{A}$ of a set $X$ a \textit{hitting set for $\mathcal{A}$} is a subset $Y$ of $X$ for which $Y\cap A\neq \emptyset$ for every $A$ in $\mathcal{A}$.

We define $\beta_{m,k}$ as the minimum real number $\beta$ for which the following holds. For any finite set $X$ and any $m$ of its subsets $A_1, \ldots, A_m$ with $|A_i|>\beta\cdot |X|$ ($i=1,2,\ldots,m$) there exists a hitting set of size at most $k$. The problem finding hitting sets has been widely studied, both theoretically and algorithmically. Perhaps the work in the literature closest to $\beta_{m,k}$ is a parameter studied by N. Alon in \cite{Alon1990}.

The following theorem is our main result. It provides a relationship between geometrical parameter $\alpha$ and the combinatorial parameter $\beta$.

\begin{theorem}
\label{thmEquiv}
	For $d$ a positive integer and $k$ an integer in $[d+1]$ we have:
	
	\[
		\alpha_{d,k}+\beta_{d+1,k}=1.
	\]
\end{theorem}

This result is useful in both directions. On the one hand, for values of $k$ close to $d$ we use the combinatorial interpretation to get sharp values for $\alpha_{d,k}$ that cannot be obtained directly from an application of Rado's centerpoint theorem.

\begin{theorem}
\label{thmExact}
For any positive integer $d$ the value of $\alpha_{d,d}$ is $\frac{d}{d+1}$.
\end{theorem}

On the other hand, for fixed values of $k$ we use Rado's centerpoint theorem to give results in hitting set theory of independent interest.

\begin{theorem}
\label{thmAppCPT}
  The value of $\beta_{m,k}$ is in $1-\Omega \left(\frac{1}{\sqrt[k]{m}}\right)$
\end{theorem}

This result turns up to be asymptotically correct if $k$ is fixed and $m$ goes to infinity.

For other values of $k$ we present a general framework to get improved bounds. The paradigm we use is the ``probabilistic method with blemishes''. We refer the reader to Chaper 3 of the excellent book \cite{Alon2011} for details. The results that can be obtained depend on the asymptotic relation between $k$ and $m$. This resembles previous applications of the probabilistic method for two-variable parameters, e.g. counting connected graphs in \cite{Hofstad2006}. As an example for the use of this framework, we study the case $m=2k$.

\begin{theorem}
\label{thmMed}
  The value of $\beta_{2k,k}$ is at most $1-\frac{1}{\sqrt[k]{15}}$ and thus the value of $\alpha_{2k-1,k}$ is at least $\frac{1}{\sqrt[k]{15}}$.
\end{theorem}

In Section \ref{secApp} we discuss two further applications of the theory developed. Here we state them briefly and we refer the reader to the corresponding section for context and details.

The first is a Helly-type theorem for fractional hyperplane tranversals through a common point.

\begin{theorem}
\label{thmTHelly}
	Let $\mathcal{F}$ be a finite family of convex sets in $\mathbb{R}^d$.

	\begin{itemize}
		\item If $\mathcal{F}$ is $k$-intersecting for $k\geq 2$ then there exists a point such that each hyperplane through it is transversal to $\Omega\left(\frac{|\mathcal{F}|}{\sqrt[k]{d+1}}\right)$ sets of the family. 
		\item If $\mathcal{F}$ is $d$-intersecting then there exists a point such that each hyperplane through it is transversal to at least $\frac{d}{d+1}\cdot|\mathcal{F}|$ sets of the family.
		\item If $\mathcal{F}$ is $\cei{\frac{d}{2}}$-intersecting then there exists a point such that each hyperplane through it is transversal to at least $\frac{1}{\sqrt[d]{225}}\cdot|\mathcal{F}|$ sets of the family.
	\end{itemize}
\end{theorem}

The second is an alternative proof of the following result of A. Holmsen \cite{Holmsen2013}.

\begin{theorem}
\label{thmHolmsen}
  Let $\mathcal{F}$ be a finite family of convex sets such that for any three distinct sets $A$, $B$, $C$ of $\mathcal{F}$ we have that
	
	\[
	  \text{conv}(A\cup B) \cap \text{conv}(B\cup C) \cap \text{conv}(C\cup A) \neq \emptyset
	\]
	
	Then there is a line transversal to at least $\frac{1}{8}|\mathcal{F}|$ sets of $\mathcal{F}$.
\end{theorem}

Finally, in Section \ref{secDiscu} we present some additional remarks and we discuss some open problems.

%%%%%%%%%%%%%%%%%%%%%%%%%
% Basic Properties
%%%%%%%%%%%%%%%%%%%%%%%%%

\section{Properties of Definition \ref{defDepth}}
\label{secProp}

\subsection{Basic properties and relation to Tukey depth}

We begin this section by showing that Definition \ref{defDepth} satisfies some expected properties for a measure of depth. For this, we follow the work made in \cite{liu1990}, where such a list of properties is established for a different notion. Some of the results in this section have rather standard proofs, so we postpone them to Section \ref{secSProofs}.

An immediate consequence of Definition $\ref{defDepth}$ is that $\dep_\mathcal{F}$ is an integer valued function and	$0\leq \dep_\mathcal{F}(p) \leq |\mathcal{F}|$. The property below reflects the intuition that if $\mathcal{F}$ has bounded sets, then points far away from the family $\mathcal{F}$ should have depth $0$.

\begin{proposition}
\label{propLim}
	For any finite family $\mathcal{F}$ of bounded convex sets in $\mathbb{R}^d$ we have that
	
	\[
		\lim_{||x||\to \infty} \dep_\mathcal{F}(x)=0.
	\]
\end{proposition}

The next property states that when $\mathcal{F}$ is symmetric about a point $p$, then the point of maximal depth should be $p$ and the depth of a point should decrease as it gets far away from $p$. To formulate this precisely, we introduce the following definition.

\begin{definition}
  We say that a family of convex sets $\mathcal{F}$ is \textit{symmetric about the origin} if for every set $F$ in $\mathcal{F}$ we have that $-F$ is also in $\mathcal{F}$.
\end{definition}

\begin{proposition}
\label{propDec}
	For any finite family $\mathcal{F}$ of convex sets in $\mathbb{R}^d$ that is symmetric about the origin and any $x\in \mathbb{R}^d$ the function $D:\mathbb{R}_{\geq 0}\to\mathbb{Z}$ given by $D(a)=\dep_{\mathcal{F}}(ax)$ is decreasing.
\end{proposition}

Notice that Theorem \ref{thmBasic} follows from Propostion \ref{propLim} and Proposition \ref{propDec}.

Now we turn to the relation between Definition \ref{defDepth} and Tukey depth \cite{Tukey1974}. Given a set of points $S$ on $\mathbb{R}^d$, the \textit{Tukey depth of a point $p$ with respect to $S$} is defined as the minimum value of $|S\cap H|$ where $H$ is a halfspace that contains $p$. We will denote this value as $\tdep_S(p)$.

The following proposition establishes some basic relationships between Tukey depth and Definition \ref{defDepth}. The first part justifies calling Definition \ref{defDepth} a generalization of Tukey depth. The second part tells us that we can get information of $\dep_\mathcal{F}$ in terms of the $\tdep$ with respect to all the representative sets of $\mathcal{F}$, but the third part tells us that is not enough to completely understand $\dep_\mathcal{F}$.

\begin{proposition}
\label{propTukey}
 \begin{itemize}
	 \item Let $S$ be a set of points in $\mathbb{R}^d$ and $\mathcal{S}$ be the family of singletons given by $S$. Then for any point $p$ we have
	
	\[
		\dep_\mathcal{S}(p)=\tdep_S(p)
	\]
	 \item Let $\mathcal{F}=\{F_1,\ldots,F_n\}$ be a family of convex sets in $\mathbb{R}^d$. Then for any point $p$ we have
	
	\[
	  \dep_\mathcal{F}(p) \geq \sup_{S:\; \forall i \,|S\cap F_i|=1}\tdep_S(p)
	\]
	 \item In some cases the inequality above is strict
 \end{itemize}
\end{proposition}

We postpone the proof of the first two statements to Section \ref{secSProofs}, but we justify the third part here.  We give an example on the plane that extends to higher dimensions. Let $abc$ be an equilateral triangle on the plane with sidelength equal to $3$ and center $o$. From it remove open disks of radius $1$ centered at each vertex. We define $\mathcal{F}$ as the family consisting of the three remaining segments. See Figure \ref{fig:fig1}.

\begin{figure}
	\centering
		\includegraphics[scale=0.6]{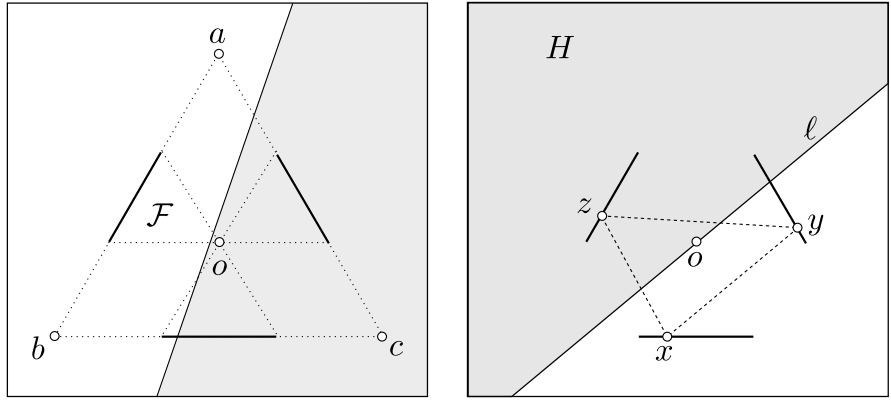}
	\caption{Example that shows that $\dep_{\mathcal{F}}(p)$ does not only depend on the Tukey depth of $p$ with respect to representative sets of the family $\mathcal{F}$}
	\label{fig:fig1}
\end{figure}
	
	Notice that any halfspace that contains $o$ always intersects two of the three segments, so according to Definition \ref{defDepth} we have $\dep_\mathcal{F}(o)=2$.
	
	On the other hand, let $P=\{x,y,z\}$ be a set of three points, one from each segment. Since the segments $xy$, $yz$ and $zx$ are not intersecting, there is one of them that does not contain $o$, say $xy$. Consider the line $\ell$ through $o$ parallel to $zy$ and $H$ the halfspace defined by $\ell$ that does not contain the segment $xy$. We have that $H$ contains $o$, but at most one of the points $x$, $y$, $z$. Therefore, the Tukey depth of $o$ with respect to $P$ is at most $1$. Thus, the supremum over all the triples $\{x,y,z\}$ is also at most $1$.
	
\subsection{Center regions and proof of Theorem \ref{thmRCen}}

We turn our attention to the level sets of $\dep_\mathcal{F}$. We define the \textit{$r$-center with respect to the family $\mathcal{F}$} as the set of points of $\mathbb{R}^d$ with depth at least $r$ with respect to $\mathcal{F}$. We denote it by $C_r(\mathcal{F})$.

The analogous definition for Tukey depth has been widely studied. In that context, it is straightforward to see that the $r$-center of a set of points $S$ is the intersection of all the closed halfspaces whose complement contains at most $r-1$ points of $S$ (see \cite{Matousek2002} for further details). Agarwal et al. \cite{Agarwal2004} studied $C_r(\mathcal{P})$ further and showed that it is a convex polytope whose faces are hyperplanes spanned by at most $d$ points of $S$ and for which the halfspace opposite to the $r$-center contains exactly $r-1$ points.

The $r$-center for Definition \ref{defDepth} is not necessarily a polytope and in such examples it cannot be expressed as the intersection of a finite number of halfspaces. Nevertheless, as stated in Theorem \ref{thmRCen}, $C_r(\mathcal{F})$ is always convex. To show this we will prove express the $r$-center as a (possibly infinite) intersection of halfspaces.

We now introduce some auxiliary definitions that will be helpful for proving Theorem \ref{thmRCen}. Let $u$ be a unit vector in $\mathbb{R}^d$ and $I$ a real interval. We define $P_{u,I}:=\{v\in \mathbb{R}^d: \langle u,v \rangle \in I\}$, and we call it the \textit{the plank with direction $u$ corresponding to $I$}. Geometrically $P_{u,I}$ is a region between two parallel hyperplanes (one or both possible at infinity).

If we project a convex set $F$ to the line through $0$ with direction $u$ then its image is an interval $I$. Therefore, any hyperplane in the plank $P_{u,I}$ is transversal to $F$. In the same spirit we present the following lemma.

\begin{lemma}
  \label{lemPlank}
  Let $n$ be a positive integer and $r$ a real number in the interval $[0,n]$. Let $\mathcal{F}$ be a family of $n$ convex sets of $\mathbb{R}^d$. Then for each direction $u$ there exists a closed plank $P$ perpendicular to it such that:
	
	\begin{itemize}
		\item For each hyperplane in the plank, each of the halfspaces it defines \emph{intersects at least}  $r$ of the sets of $\mathcal{F}$.
		\item For each of the two bounding hyperplanes of the plank, the halfspace defined by it that contains the plank, \emph{contains more} than $n-r$ sets of $\mathcal{F}$.
	\end{itemize}
\end{lemma}

\begin{proof}
  For each real number $x$ define $f^+(x)$ as the number of sets of $\mathcal{F}$ that intersect the halfspace $H^+_x:=\{v\in \mathbb{R}^d: \langle u,v \rangle \geq x\}$ and $f^-$ as the number of sets of $\mathcal{F}$ that intersect the halfspace $H^-_x:=\{v\in \mathbb{R}^d: \langle u,v \rangle \leq x\}$. See Figure \ref{fig:fig2}. Define $f:=\min(f^+,f^-)$. Since $f^+$ is decreasing and $f^-$ is increasing  the function $f$ is unimodal, concave, with minimum $0$ and bounded above by $n$. Therefore, $f^{-1}([r,n])$ is a closed interval $[a,b]$. By definition, for each hyperplane in the plank $P_{u,[a,b]}$ each of the halfspaces it defines intersect at least $r$ sets of $\mathcal{F}$.
	
 To finish the proof, consider one of the bounding hyperplanes of the plank, say the one through $au$. By closedness, for a small value of $\epsilon$, the halfspaces $H^{+}_a$ and $H^{+}_{a-\epsilon}$ contain the same number of sets of $\mathcal{F}$. Since $H^{-}_{a-\epsilon}$ intersects less than $r$ sets of $\mathcal{F}$, then $H^{+}_{a-\epsilon}$ (and thus also $H^+_a$) contains more than $n-r$ sets of $\mathcal{F}$.
	
\begin{figure}
	\centering
		\includegraphics[scale=0.8]{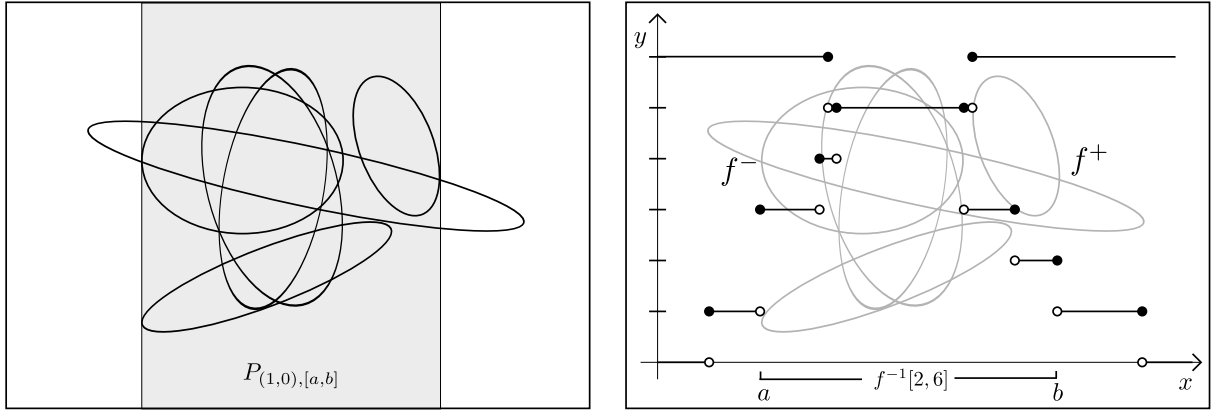}
	\caption{Example of projection of $6$ convex sets to the $x$-axis. The image on the left shows the plank corresponding to points with depth at least $2$ in direction $x$. The image on the right shows the functions $f^-$ and $f^+$ overlapped on the family $\mathcal{F}$.}
	\label{fig:fig2}
\end{figure}
\end{proof}

Let $\mathcal{P}_r(\mathcal{F})$ denote the family of planks obtained by varying the direction $u$ in Lemma \ref{lemPlank}. The following lemma relates this family to the $r$-center with respect to $\mathcal{F}$.

\begin{lemma}
\label{lemInt}
	For any family $\mathcal{F}$ of convex sets we have that
	
	\[
	C_r(\mathcal{F})=\bigcap \mathcal{P}_r(\mathcal{F})
	\]
	
\end{lemma}

\begin{proof}
  First let $p$ be a point in $\bigcap\mathcal{P}_r(\mathcal{F})$ and $H$ a closed halfspace that contains $p$ with bounding hyperplane $\Pi$. Let $u$ be a unit vector perpendicular to $\Pi$ and $P$ the plank given by Lemma \ref{lemPlank}. By hypothesis, $p$ is in $P$. The halfspace $H$ contains one of the halfspaces $H'$ defined by the hyperplane through $p$ perpendicular to $u$. By the first part of the lemma, $H'$ contains at least $r$ sets from $\mathcal{F}$ and thus the same is true for $H$. This proves that $p$ is in $C_r(\mathcal{F})$.

	Now let $p$ be a point that is not in $\bigcap\mathcal{P}_r(\mathcal{F})$. This means that for some unit vector $u$ the point $p$ is not in the plank $P$ given by Lemma \ref{lemPlank}. Let $\Pi$ be the bounding hyperplane of $P$ closest to $p$ and $\Pi'$ a hyperplane that separates $p$ from $\Pi$. By the second part of the lemma, the closed halfspace defined by $\Pi$ that contains $P$ has more than $n-r$ sets of $\mathcal{F}$ and therefore the closed halfspace defined by $\Pi'$ that contains $p$ has less than $r$ sets. This shows that $p$ is not in $C_r(\mathcal{F})$.
	
\end{proof}

We are ready to prove Theorem \ref{thmRCen}.

\begin{proof}[Proof of Theorem \ref{thmRCen}]
  By Lemma \ref{lemPlank} and Lemma \ref{lemInt} we have that the $r$-center is the intersection of planks. Each plank is convex, and thus the $r$-center is convex.
\end{proof}

%%%%%%%%%%%%%%%%%%%%%%%%%
% Centerpoint results
%%%%%%%%%%%%%%%%%%%%%%%%%

\section{Intersection patterns and centerpoint results}
\label{secCPT}

One of the cornerstones of discrete geometry is Rado's centerpoint theorem. In the terms used above, it states that for a set of points $P$ there always exists a point with Tukey depth at least $\frac{|P|}{d+1}$. Moreover, this result is optimal as there are sets of points with no point with Tukey depth greater than this number. As an easy consequence of this result and Proposition \ref{propTukey} we have that the analog result is also true for Definition \ref{defDepth}.

The aim of this section is to show that this basic bound can be sharpened by knowing more information on the intersection pattern of $\mathcal{F}$. Let us give a trivial but illustrative example. Suppose that we know that the familily $\mathcal{F}$ is $(d+1)$-intersecting. Then by Helly's theorem the whole family $\mathcal{F}$ has non-empty intersection. If $p$ is any point in the intersection then any closed halfspace that contains $p$ intersects every set of $\mathcal{F}$. Thus we have managed to improve the best possible value of $\dep_\mathcal{F}$ from $\frac{|\mathcal{F}|}{d+1}$ to $|\mathcal{F}|$.

Intuitively, if the convex sets of the family intersect more, then we should be able to find points with higher depth. To formalize this we introduce the following definition.

\begin{definition}
	Let $d$ be a positive integer and $k$ an integer in $[d]$. We define $\alpha_{d,k}$ as the largest real number such that any finite family $\mathcal{F}$ of $k$-intersecting convex sets in $\mathbb{R}^d$ has a point $p$ for which $\dep_\mathcal{F}(p) \geq \alpha_{d,k}\cdot |\mathcal{F}|$.
\end{definition}

The discussion above shows that $\alpha_{d,1}=\frac{1}{d+1}$ and $\alpha_{d,d+1}=1$. For general values of $d$ and $k$ we have the following lower bound for $\alpha_{d,k}$.

\begin{proposition}
\label{propBBound}
	Let $\mathcal{F}$ be a finite family of $k$-intesecting convex sets in $\mathbb{R}^d$. Then there exists a point $p$ whose depth $r:=\dep_\mathcal{F}(p)$ satisfies:
	
	\[
		\binom{r}{k}\geq \frac{1}{d+1}\binom{|\mathcal{F}|}{k}.
	\]
\end{proposition}

\begin{proof}
  For each collection $\mathcal{G}$ of $k$ sets from $\mathcal{F}$ we choose a point in its intersection. This gives us a set $S$ of $\binom{|\mathcal{F}|}{k}$ points. We apply Rado's centerpoint theorem to $S$ to obtain a point $p$ with Tukey depth at least $\frac{1}{d+1}\binom{|\mathcal{F}|}{k}$.
	
	Let $H$ be a closed halfspace that contains $p$ and intersects the minimal number $r$ of sets from $\mathcal{F}$. Then it can contain at most $\binom{r}{k}$ points from $S$. Since $p$ is a Tukey centerpoint for $S$ we have:
	
	\[
		\binom{r}{k}\geq |S\cap H| \geq \frac{1}{d+1}\binom{|\mathcal{F}|}{k}.
	\]
\end{proof}

\begin{proof}[Proof of Theorem \ref{thmAsy}]
We use Proposition \ref{propBBound} and standard bounds on binomial coefficients:

\[
	\frac{r^k}{k!}\geq \binom{r}{k}\geq \frac{1}{d+1}\binom{|\mathcal{F}|}{k}\geq \frac{1}{d+1}\cdot \frac{(|\mathcal{F}|-k+1)^k}{k!}
\]

\noindent and thus

\[
	\frac{r}{|\mathcal{F}|}\geq \frac{1}{\sqrt[k]{d+1}} \cdot \left(1-\frac{k-1}{|\mathcal{F}|}\right) = \Omega\left(\frac{1}{\sqrt[k]{d+1}}\right) 
\]

\end{proof}

\subsection{A related hitting-set problem}

The bound given in Proposition \ref{propBBound} is useful in some situations, but it is not sharp. As an example, for pairwise intersecting sets on the plane it gives the bound $\alpha_{2,2}\geq \frac{1}{\sqrt{3}}$, but as we will see later, we actually have $\alpha_{2,2}=\frac{2}{3}$. Our main tool to give better bounds for $\alpha_{d,k}$ is a key relationship to a purely combinatorial parameter.

\begin{definition}
	Let $m$ be a positive integer and $k$ an integer in $[m]$. We define $\beta_{m,k}$ as the smallest real number $\beta$ for which the following holds. For any finite set $X$ and any $m$ of its subsets $A_1, \ldots, A_m$ with $|A_i|>\beta\cdot |X|$ ($i=1,2,\ldots,m$) there exists a hitting set of size at most $k$.	
\end{definition}

Theorem $\ref{thmEquiv}$ establishes the relationship between $\alpha$ and $\beta$. Before we prove it, we present an auxiliary geometric result on non-intersecting planks. It does not require the $k$-intersecting hypothesis. See Figure \ref{fig:fig3} for an accompanying example on the plane.

\begin{proposition}
  \label{propSimplex}
	Let $\mathcal{F}$ be a finite family of convex sets in $\mathbb{R}^d$ and $r$ an integer in $[|\mathcal{F}|]$. Then either
	
	\begin{itemize}
		\item there exists a point with depth at least $r$ or
		\item we can find $d+1$ halfspaces with empty intersection each of which contains more than $|\mathcal{F}|-r$ sets of $\mathcal{F}$.
	\end{itemize}
\end{proposition}

\begin{figure}
	\centering
		\includegraphics[scale=0.7]{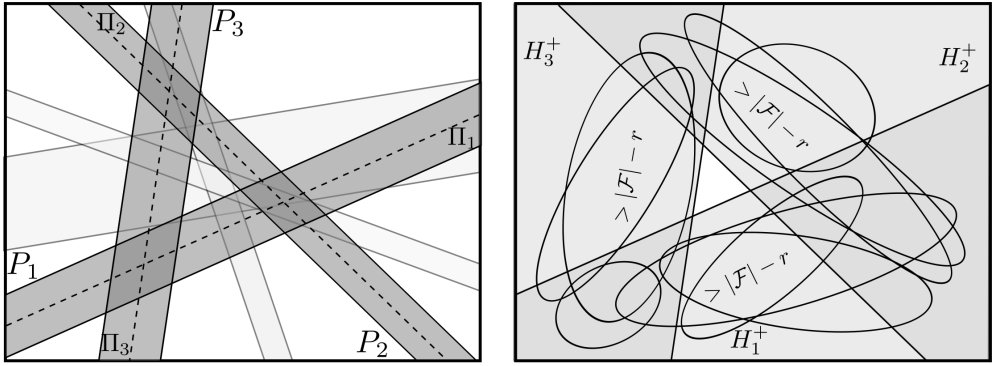}
	\caption{Finding non-intersecting hyperplanes, each containing more than $|\mathcal{F}|-r$ sets of $\mathcal{F}$}
	\label{fig:fig3}
\end{figure}

\begin{proof}
  We define $\mathcal{P}_r(\mathcal{F})$ as in Section \ref{secProp}. Suppose for a moment that $\mathcal{P}_r(\mathcal{F})$ is $(d+1)$-intersecting. By Helly's theorem, any finite collection of $\mathcal{P}_r$ would be intersecting. By a compactness argument, the whole family would be intersecting. Therefore, by Lemma \ref{lemInt} the center region $C_r(\mathcal{F})$ would be non-empty. Thus if the first conclusion is not satisfied it is because we can find $d+1$ planks with empty intersection. Call them $P_1$, $P_2$, $\ldots$, $P_{d+1}$. Using that the planks are closed, we may assume that their defining directions are affinely independent.
	
	For each plank $P_i$ let $\Pi_i^+$ and $\Pi_i^-$ be its boundary hyperplanes, $\Pi_i$ the plane that lies in between them exactly at the middle and let $H_i^+$ (resp. $H_i^-$) be the halfspace with boundary $\Pi_i^+$ (respectively $\Pi_i^-$) that contains $\Pi_i$. Notice that $P_i=H_i^+\cap H_i^-$, so again by Helly's theorem a subset of at most $d+1$ of these halfspaces must have empty intersection. Now, by the affine independence of the directions, the intersection of any $d$ planes $\Pi_i$ is non-empty. Therefore, the $(d+1)$-family of non-intersecting halfspaces has exactly one halfspace for each index $i$. By relabeling we may assume that $\cap_{i\in[d+1]}H_i^+$ is empty.
	
  We end the proof by noting that by the second part of Lemma \ref{lemPlank} each halfspace $H_i^+$ contains more than $|\mathcal{F}|-r$ sets of $\mathcal{F}$.
\end{proof}

We are ready to prove the key relation between $\alpha$ and $\beta$.

\begin{proof}[Proof of Theorem \ref{thmEquiv}]
	We first show that $\alpha_{d,k}\geq 1-\beta_{d+1,k}$. We proceed by contradiction. Thus lets assume that there exists a finite family $\mathcal{F}$ of convex sets in $\mathbb{R}^d$ that is $k$-intersecting with no point of depth $1-\beta_{d+1,k}$. Let $H_1, \ldots, H_{d+1}$ be the family of halfspaces given by Proposition \ref{propSimplex}. For each $i \in [d+1]$ we define $A_i$ as the set of convex sets contained in $H_i$. By Proposition \ref{propSimplex}, each $A_i$ has size larger than
	
	\[
		|\mathcal{F}|-(1-\beta_{d+1,k})\cdot |\mathcal{F}| = \beta_{d+1,k} \cdot |\mathcal{F}|.
	\]
	
	By the definition of $\beta_{d+1,k}$, there is a hitting set for $\{A_1,A_2,\ldots,A_{d+1}\}$ of size at most $k$. In other words, we can find at most $k$ sets of $\mathcal{F}$ whose intersection is completely contained in
	
	\[
		H_1\cap H_2 \cap \ldots \cap H_{d+1}=\emptyset.
	\]
	
	This is a contradiction to the $k$-intersecting hypothesis.
	
	Now we show that $\beta_{d+1,k}\leq 1-\alpha_{d,k}$. Once again, we proceed by contradiction and suppose that there exists a finite set $X$ with $d+1$ subsets $A_1,\ldots,A_{d+1}$ such that every one of them has size larger than $(1-\alpha_{d,k})\cdot |X|$ but with no hitting set of size $k$ for $\mathcal{A}:=\{A_1,\ldots,A_{d+1}\}$. From here we will construct a family $\mathcal{F}$ of convex sets in $\mathbb{R}^d$.
	
	We begin with a set of $d+1$ points in general position in $\mathbb{R}^d$, say $\mathcal{P}=\{P_1, P_2, \ldots, P_{d+1}\}$. For each $i\in [d+1]$ we define $C_i$ as the convex hull of $\mathcal{P}\setminus \{P_i\}$.  Note that the sets $C_i$ are $d$-intersecting, but the intersection of all of them is empty.	Now, for each $x\in X$ we set $I_x$ as the family of indexes $i$ for which $x\in A_i$ and
	
	\[
		F_x = \bigcap_{i\in I_x} C_i.
	\]
	
	Finally, we set $\mathcal{F}:=\{F_x\}_{x\in X}$. Each set $F_x$ is the intersection of convex sets, and thus is convex (it is actually a face of the simplex spanned by $\mathcal{P}$). We claim that $\mathcal{F}$ is $k$-intersecting. Indeed no $k$-set $Y$ of $X$, is a hitting set for $\mathcal{A}$. Then:
	
	\[
		\bigcap_{x\in Y} F_x = \bigcap_{i:\; i\in I_x \text{ for some } x\in Y} C_i \supsetneq \bigcap_{i\in [d+1]} C_i = \emptyset.
	\]
	
	So we may use the definition of $\alpha_{d,k}$ and find a point $p$ of depth at least $\alpha_{d,k}$ with respect to $\mathcal{F}$. As stated above, the sets $C_i$ have empty intersection, so without loss of generality we may suppose $p$ is not in $C_1$. Then, we can find a closed halfspace $H$ that does not intersect $C_1$ and whose boundary hyperplane $\Pi$ contains the point $p$. All the sets $F_x$ with $x\in A_1$ are contained in $C_1$. By hypothesis, $|A_1|> (1-\alpha_{d,k})\cdot |\mathcal{F}|$. Then the halfspace $H$ intersects less than
	
	\[
		|\mathcal{F}|-(1-\alpha_{d,k})\cdot |\mathcal{F}| = \alpha_{d,k} \cdot |\mathcal{F}|
	\]
	
	\noindent sets of $\mathcal{F}$. This contradiction to the definition of $\alpha_{d,k}$ finishes the proof.
\end{proof}
\subsection{Exact values and better bounds for $\beta_{m,k}$}

The follow proposition gives the exact value for $\beta_{k+1,k}$. Notice that by Theorem \ref{thmEquiv}, this proves Theorem \ref{thmExact}.

\begin{proposition}
	The value of $\beta_{k+1,k}$ is $\frac{1}{k+1}$.
\end{proposition}

\begin{proof}

If we have $k+1$ subsets of a finite set $X$ and each of them has more than $\frac{1}{k+1}\cdot |X|$ elements, then two of them must intersect, and then by choosing a point in the intersection and $k-1$ arbitrary points, one from each of the remaining sets, we get a hitting set with $k$ elements. To show that this is optimal, we set $X=[k+1]$ and for each $i\in[k+1]$ we set $A_i=\{i\}$. Each family has exactly $\frac{1}{k+1}\cdot |X|$ of the elements of $X$ but there is no hitting set of size $k$ to the family.

\end{proof}

 A general bound for $\beta_{m,k}$ (Theorem \ref{thmAppCPT}) follows immediately from Proposition \ref{propBBound} and Theorem \ref{thmEquiv}. If $k$ is constant and $m$ is large this bound is in $1-\Omega\left(\frac{1}{\sqrt[k]{m}}\right)$. This is the best bound that we can expect asymptotically. Fix a positive integer $n$ larger than $k$ and consider the family $\mathcal{A}$ of all $(n-k)$-subsets of $n$. This is a family of $m=\binom{n}{k}\leq \frac{n^k}{k!}$ sets, with no hitting set of size $k$ and each of which the proportion of its size to $n$ is

\[
  \frac{n-k}{n}=1-\frac{k}{n}\geq 1-\frac{k}{\sqrt[k]{k!m}}=1-O\left(\frac{1}{\sqrt[k]{m}}\right).
\]

We have discussed the case in which $k$ is close to $m$ and in which $k$ is fixed and $m$ is large. What happens when $k$ lies in between? We will now present a framework to obtain better bounds. We assume that the reader is familiar with probabilistic arguments.

We are given a set $X$ and a family of $m$ subsets $\{A_1,\ldots,A_m\}$. We work under the assumption that that each set $A_i$ satisfies $|A_i|\geq \beta\cdot |X|$, and at the end we determine the required value of $\beta$. If we construct a set $Y$ by taking $k$ elements uniformly and independently from $X$, the probability that $Y$ does not intersect a set $A_i$ is $(1-\beta)^k$. So by the union bound, the probability of $Y$ not being a hitting set is bounded by $m(1-\beta)^k$. If this quantity is less than one, then the probability of getting a hitting set is positive and we are done. So it suffices to set $\beta>1-\frac{1}{\sqrt[k]{m}}$. But this is asymptotically the same bound that we got using Rado's centerpoint theorem.

To improve this value, we allow some ``blemishes'' to happen. Instead of constructing $Y$ with $k$ random elements, we only sample $k-\ell$ for some $\ell\in[k]$. We work with an arbitrary value of $\ell$ and at the end we modify it to optimize our argument. The probability that $Y$ hits a set is $1-(1-\beta)^{k-\ell}$, and thus in expectation we hit $m(1-(1-\beta)^{k-\ell})$ sets. If 

\begin{align}
\label{ecExp}
m(1-(1-\beta)^{k-\ell})>m-\ell-1,
\end{align}

\noindent then there is an instance of $Y$ that hits at least $m-\ell$ sets and by (deterministically) choosing one point from the remaining $\ell$ sets we get a hitting set for $\mathcal{A}$ with $k$ elements.

So if we fix $m$ and $k$, we would like to apply the argument above to the value of $\ell$ whose solution in $\beta$ is minimal. In general, this raises a difficult optimization problem. But let us illustrate how the argument can improve our estimations when $k$ is $m/2$.

\begin{proof}[Proof of Theorem \ref{thmMed}]
  We have $m=2k$. We use the strategy above for $\ell=0.37 k$. We have to show that inequality (\ref{ecExp}) holds for $\beta=1-\frac{1}{\sqrt[k]{15}}$. Notice that it is equivalent to:
	
	\[
		\frac{\ell+1}{m}> (1-\beta)^{k-\ell}
	\]
	
	On the left hand side we have a quantity larger than $$\frac{\ell}{2k}=\frac{0.37k}{2k}> 0.185.$$
	
	On the right hand side we have $$\left(\frac{1}{\sqrt[k]{15}}\right)^{0.63k}=\frac{1}{15^{0.63}}<0.182.$$
	
	This finishes the proof

\end{proof}

%%%%%%%%%%%%%%%%%%%%%%%%%
% Helly-type theorem
%%%%%%%%%%%%%%%%%%%%%%%%%

\section{A Helly-type fractional transversal theorem}
\label{secApp}

One possible direction to generalize Helly's theorem is to find hypothesis that guarantee a hyperplane transversal through all the sets. Hadwiger's celebrated theorem \cite{Hadwiger1957} states that if we have a labeled family of convex sets $\mathcal{F}=\{F_1,\ldots,F_n\}$ on the plane and for each ordered triple $i<j<k$ there is a line transversal to $F_i$, $F_j$, $F_k$ in that order, then there is a line transversal to all the sets of $\mathcal{F}$. This theorem was subsequently generalized to higher dimensions \cite{Pollack1989,JacobE.Goodman1988} in terms of order types and furthermore it has even been generalized to colorful versions \cite{Holmsen2016,Arocha2009,Arocha2008}.

With non-order-type transversal hypothesis usually the best that one can expect is a hyperplane transversal to many but not all of the sets of the family. One of the possible directions is to require that each $k$ sets of the family have a common transversal (the \textit{$T(k)$ property}). Under this assumption, Katchalski and Liu \cite{katchalski1980} showed on the plane the existance of fractional hyperplane transversals, that is, a transversal through a positive fraction of the members of the family. The $T(3)$ property on the plane has been widely studied on the plane and an account of the most recent contributions can be found in the nice survey by Holmsen \cite{Holmsen2013} and the references therein.

Another possibility to obtain fractional transversals is to require a $k$-intersecting hypothesis. For $k=2$ a standard projection argument shows that any pairwise and finite family of convex sets has in every direction a transversal hyperplane to all the sets. The following result is trade-off variant: the hyperplanes intersect only a positive fraction of the sets, but they all go through a common point.

\begin{proposition}
\label{propTrans}
  Let $d$ be a positive integer and $k$ an integer in $\{2,\ldots,d+1\}$. Let $\mathcal{F}$ a finite family of $k$-intersecting convex sets in $\mathbb{R}^d$. Then there exists a point such that any hyperplane through it is transversal to at least $\alpha_{d,k}\cdot|\mathcal{F}|$ sets of $\mathcal{F}$.
\end{proposition}

\begin{proof}
	Let $p$ be a point with depth at least $\alpha_{d,k}\cdot |\mathcal{F}|$. This is the required point. Indeed, suppose that it is not. Then there exists a hyperplane $\Pi$ through $p$ that is transversal to less than $\alpha_{d,k}\cdot |\mathcal{F}|$ sets of the family. But then, by the definition of $\alpha_{d,k}$, each open halfspace defined by $\Pi$ must contain a set from $\mathcal{F}$. By taking one set on each side we get two sets that do not intersect. Since $k\geq 2$, this is a contradiction to the $k$-intersecting hypothesis.
\end{proof}

As a corollary to Proposition \ref{propTrans} and the bounds on $\alpha$, we obtain Theorem \ref{thmTHelly}, a Helly-type result for fractional transversals. In comparison to previous results, no such conclusion can be obtained from the $T(k)$ property. Indeed, if we take $k$ points in a line they satisfy the $T(k)$ property, but for most directions a hyperplane can only contain one of these points.

We end this section with a cute application: an alternative proof for Theorem \ref{thmHolmsen}. A key ingredient in the original proof \cite[Lem. 3]{Holmsen2013} is to see that under the hypothesis there is a pairing of the elements of $\mathcal{F}$ and a point $p$ such that $p$ is in the convex hull of $A\cup B$ for each paired sets $A$ and $B$. This is an interesting partition result and its proof requires a careful and non-trivial analysis of cases. But we can shortcut around it by replacing $p$ with a deep enough point.

\begin{proof}
We claim that under the hypothesis there is a point $p$ with depth at least $\frac{1}{2}\cdot|\mathcal{F}|$ with respect to $\mathcal{F}$. Indeed, if this was not the case by Proposition \ref{propSimplex} we would have three non-intersecting halfspaces $H_1$, $H_2$, $H_3$ each containing strictly more than $\frac{1}{2}\cdot|\mathcal{F}|$ sets of $\mathcal{F}$. Then $H_1\cap H_2$, $H_2\cap H_3$, $H_3\cap H_1$ each contains at least one set of $\mathcal{F}$. By taking one set on each of these intersections we contradict triples hypothesis.

We proceed as in the original proof (see \cite[Proof of Thm. 1]{Holmsen2013}) and define for each $X$ in $\mathcal{F}$ the double cone $\{L_A\}$, the family $M$ and the pair $A$, $B$ of sets in $\mathcal{F}$ whose cones have maximal angular distance. 

Now we refer the reader to \cite[Fig. 5]{Holmsen2013}. The half-plane defined by $\ell_2$ that does not contain $A$ and $B$ must intersect at least half of the sets of $\mathcal{F}$. But none of these sets can be contained in any of the regions defined by the lines $\ell_i$, since it would violate either the maximality of $A$ and $B$, the minimality of $M$ or the tight triple hypothesis. Therefore, at least $\frac{1}{2}$ of the sets are intersected by the lines $\ell_i$ and by the pidgeon-hole principle one of these lines must intersect $\frac{1}{8}\cdot|\mathcal{F}|$ elements of $\mathcal{F}$.

\end{proof}

%%%%%%%%%%%%%%%%%%%%%%%%%
% Discussion
%%%%%%%%%%%%%%%%%%%%%%%%%

\section{Discussion and open problems}
\label{secDiscu}

\begin{itemize}
	\item We have provided a probabilistic framework to give upper bounds for $\beta_{m,k}$, and thus lower bounds for $\alpha_{d,k}$. The main open problem that we leave is the following.
	
	\begin{problem}
	Give a detailed analysis on the results of this method depending on the asymptotic relation between $m$ and $k$.
	\end{problem}	
	\item There has been some interest in constructing centerpoints and center regions for Tukey depth algorithmically. See for example Agarwal et al. \cite{Agarwal2004} and the references therein. Definition \ref{defDepth} deals with more complicated objects, so in principle one would expect that the algorithms for finding centerpoints and center regions are more complex. In view of this, we raise the following problem.
	
	\begin{problem}
		What role play the properties of the family $\mathcal{F}$ in the complexity of finding centerpoints or center regions for Definition \ref{defDepth}?
	\end{problem}
	
	\item There are other results in the literature that implicitly relate the intersection patterns of a family of convex sets with an underlying notion on depth. For example, Rousseeuw and Hubert \cite{Rousseeuw1999} conjectured that in any arrangement of $n$ hyperplanes in general position in $\mathbb{R}^d$ there exists a point such that any ray that starts on it intersects at least $\cei{n/(d+1)}$ hyperplanes. This was later proved by Amenta et al. \cite{Amenta2000}. Note that such an arrangement of planes is $d$-intersecting but as far as we know, there is no evident relation between this problem and points with high depth for Definition \ref{defDepth}.
	\item Another problem with a similar flavor was posed by Jorge Urrutia: What is the smallest number $t=t(n)$ such that for any family of $n$ \textit{pairwise disjoint} segments on the plane there exists a point such that any ray that starts on it intersects at most $t$ segments? Fulek et al. \cite{Fulek2009} studied a more general version of the problem and in particular showed that the value of $t(n)$ is roughly $\frac{2n}{3}$.
\end{itemize}

\section{Appendix: Other proofs}
\label{secSProofs}

\begin{proof}[Proof of Proposition \ref{propLim}]
	Since $\mathcal{F}$ is a finite family of bounded sets, its union is also bounded. Let $B$ be a closed ball centered at the origin and that strictly contains $\cup \mathcal{F}$, and call its radius $r$. If $||x||>r$, then by the separation theorem there is a hyperplane $H$ that separates $x$ and $B$. The halfspace defined by $H$ that contains $x$ intersects no set from $\mathcal{F}$. Therefore, $\dep_\mathcal{F}(x)=0$.
\end{proof}

\begin{proof}[Proof of Proposition \ref{propDec}]
	By symmetry and by Theorem \ref{thmRCen}, each $r$-center with respect to $\mathcal{F}$ is a symmetric convex set centered at the origin, and in particular it contains the origin. Consider any positive real number $b$ and a real number $a$ in $[0,b]$. Both $0$ and $bx$ are in the $D(b)$-center with respect to $\mathcal{F}$. By convexity, $ax$ is also in the $D(b)$-center. This implies $D(a)\geq D(b)$, as desired.	
\end{proof}

\begin{proof}[Proof of Proposition \ref{propTukey}]
\begin{itemize}
	\item This part follows from the definitions and from the fact that a halfspace intersects a singleton if and only if it contains its corresponding point.
	\item We fix a set $\mathcal{S}$ such that $|S\cap F_i|=1$ for every index $i\in [n]$. We know that each halfspace that contains $p$ contains at least $\tdep_S(p)$ points of $S$, so in particular it intersects $\tdep_S(p)$ sets of $\mathcal{F}$. This shows that $\dep_\mathcal{F}(p)\geq \tdep_S(p)$. By taking the supremum on both sides of this inequality over all such sets $S$ we get the desired result.
\end{itemize}
\end{proof}

\section{Acknowlegments}
The project leading to this application has received funding from European Research Council (ERC) under the European Union’s Horizon 2020 research and innovation programme under grant agreement No. 678765 and from the Israel Science Foundation under grant agreement No. 1452/15.

\bibliographystyle{spmpsci}      % mathematics and physical sciences
\bibliography{CPHelly}   % name your BibTeX data ba
   
\end{document}